\documentclass[12pt]{amsart}
\usepackage{amssymb, amsmath, amsthm, enumerate}

\newcommand{\Pb}{\mathbb{P}}

\newcommand{\N}{\mathbb{N}}

\newcommand{\R}{\mathbb{R}}

\usepackage{geometry}

\theoremstyle{plain}
\newtheorem{thm}{Theorem}[section]
\newtheorem{lem}[thm]{Lemma}
\newtheorem{cor}[thm]{Corollary}
\newtheorem{prop}[thm]{Proposition}
\theoremstyle{definition}

\newtheorem{rem}[thm]{Remark}
\newtheorem*{problem}{Problem}

\begin{document}
\title{Sign changes of Hecke eigenvalues}
\author{Kaisa Matom\"aki}
\address{Department of Mathematics and Statistics \\ University of Turku \\ 20014
  Turku \\ Finland}
\email{ksmato@utu.fi} 
\thanks{The first author was supported by the Academy of Finland grant no. 137883. The second author was partially supported by NSF grant DMS-1128155}

\author{Maksym Radziwi\l\l}
\address{School of Mathematics \\ Institute for Advanced Study \\ 1 Einstein
Drive \\ Princeton, NJ, 08540, USA}
\email{maksym@ias.edu}
\subjclass[2000]{}
\curraddr{Department of Mathematics,
Rutgers University \\
Hill Center - Busch Campus,
110 Frelinghuysen Road \\
Piscataway, NJ 08854-8019, USA}
\email{maksym.radziwill@gmail.com}
\date{\today}



\maketitle

\section{Introduction}
\label{sec:intro}
Let $f$ be a holomorphic Hecke cusp form or a Maass Hecke cusp form
and write $\lambda_f(n)$
for its normalized Fourier coefficients, so that
$$
f(z) = \sum_{n = 1}^{\infty} \lambda_f(n) n^{(\kappa - 1)/2} e(n z)
$$
in the holomorphic case, and
$$
f(z) = \sqrt{y} \sum_{n \neq 0}  \lambda_f(n) K_{s-1/2}(2\pi |n| y)
e(n x)
$$
in the Maass case. 

In~\cite{Kowalski} Kowalski, Lau, Soundararajan and Wu investigate the problem of the first sign change of $\lambda_f(n)$
for holomorphic $f$. They remark on the similarities with the problem of the least quadratic 
residue.
This motivates the point of view that the signs of $\lambda_f(n)$ are $\text{GL}(2)$ analogues of real characters. 
The frequency of signs and sign changes and other related questions
have been also recently studied in \cite{ChSaSe09, LauLiuWu10, LauWu09, MeShVi13, WuZhai13}. In \cite{GhoshSarnak} Ghosh and Sarnak 
relate sign changes of $\lambda_f(n)$ to zeros of $f(z)$ on the imaginary line. 

The sequence of signs of $\lambda_f(n)$ alone is known to determine the form $f$ (see \cite{Kowalski}). 
Because of this we expect a significant amount of randomness 
in the distribution of the signs of $\lambda_f(n)$. It is the
finer details of this randomness that we set out to investigate in this
paper.
We start by showing that half of the non-zero $\lambda_f(n)$ are positive
and half are negative. This is relatively straightforward and depends
only on the multiplicativity of $\lambda_f(n)$. For simplicity
we focus on the case of the full modular group.
%
\begin{thm}
\label{th:longintf}
Let $f$ be a holomorphic Hecke cusp form for the full modular group 
Asymptotically half of the non-zero $\lambda_f(n)$ are positive and half of them are negative. 
\end{thm}
Theorem 1.1 was also recently obtained independently (for Maass forms)
by Elliott and Kish \cite[Theorem 7]{ElliottKish}. 
Previously it was known (for holomorphic forms) that a positive proportion of the $\lambda_f(n)$ are positive and a positive proportion negative \cite[Theorem 1]{LauWu09}.
If  $\lambda_f(n)$ and $\lambda_f(n + 1)$ behave
independently, and $\lambda_f(n)$ never vanishes then 
we expect $(1/2 + o(1)) x$ sign changes in the sequence $(\lambda_f(n))_{n \leq x}$, since individually
each  $\lambda_f(n)$ and $\lambda_f(n + 1)$ is positive and negative half of the time. 
Even if $\lambda_f(n)$ happened to vanish we still expect a positive
proportion of sign changes on the relative subset of those $n$ for which
$\lambda_f(n) \neq 0$. 
Note that we do not consider the situation where $\lambda_f(n) < 0$ 
and $\lambda_f(n + 1) = 0$ to be a sign change, because the sign of $0$ is
undefined. Thus, by number of sign changes of a possibly vanishing
sequence $\lambda_f(n)$ we mean the number of sign changes of $\lambda_f(n)$
on the subset of $n$ at which $\lambda_f(n) \neq 0$. 
Our main result is the following.
%
\begin{thm}
\label{th:signchanges}
Let $f$ be a holomorphic Hecke cusp form for the full modular group or a Hecke Maass cusp form for the full modular group.
Then for every large enough $x$, the number of sign changes in the sequence $(\lambda_f(n))_{n \leq x}$
is of the order of magnitude
\begin{equation}
\label{eq:lafnneq0form}
 \#\{n \leq x: \lambda_f(n) \neq 0\} \asymp
x \prod_{\substack{p \leq x \\ \lambda_f(p) = 0}} \Big ( 1 - \frac{1}{p} \Big ) . 
\end{equation}
\end{thm}
For holomorphic forms $f$ Serre \cite{Serre81} established that $\lambda_f(n) \neq 0$ for a positive proportion of $n$. This is expected to hold for Hecke Maass forms of eigenvalue $> \tfrac 14$ but is not currently known. By comparing the second and fourth moment of $|\lambda_f(n)|$ with $f$ a Hecke Maass form, one can show that $\lambda_f(n) \neq 0$ for at least $c X / \log X$ integers $n \leq X$. Slightly better results can be obtained by comparing moments of $|\lambda_f(p)|$.

Previously it was only known that there are $\gg x^{1/2}$ sign changes in the holomorphic case (see \cite{LauWu09})
and $x^{1/8-\varepsilon}$ sign changes in the Maass case (see \cite{ChSaSe09}).
We will prove in Section~\ref{se:pfofcor} the following corollary on a variant of Chowla's conjecture.
\begin{cor}
\label{cor:Cho} Let $f$ be a holomorphic Hecke cusp form for the full modular group. There exists a constant $c >0$ such that, for all $x$ large enough,
\[
\Big | \sum_{n \leq x} \text{sgn}(\lambda_f(n)) \text{sgn}(\lambda_f(n + 1)) \Big |
\leq (1 - c) x.
\]
\end{cor}
This is non-trivial when all of the $\lambda_f(n)$ are non-zero.
It would be interesting to rule out the possibility of $\lambda_f(n)$ regularly flipping its sign and thus to investigate the \textit{entropy} of
the sequence of sign of $\lambda_f(n)$. 

According to the Sato-Tate conjecture (now a theorem), for any $I \subset [0,2]$ there is a positive proportion of primes
$p$ such that $|\lambda_f(p)| \in I$. Because of this the size of $|\lambda_f(n)|$ fluctuates
wildly, making it hard to detect sign changes.  This is similar to a problem one encounters when studying the zeros of the Riemann zeta-function on the half-line. 
In analogy to that problem, we introduce sieve weights $w_n \geq w_n' \geq 0$ and compare
\begin{equation}
\label{eq:wnwn'sums}
\Big | \sum_{x \leq n \leq x + h} \text{sgn}(\lambda_f(n)) w_n \Big | \text{ and } \sum_{x \leq n \leq x + h} w_n'
\end{equation}
with the intent of showing that the first term is frequently less than the second (since this ensures a
sign change in the interval $[x, x + h]$).
To obtain cancellations in the first sum in \eqref{eq:wnwn'sums} (for almost all $x$), we 
use the work of Harcos on the shifted
convolution problem \cite{Harcos03},
\begin{equation} \label{eq:scp}
\sum_{\substack{X \leq m , n \leq 2X \\ am - b n = h}} \lambda_f(n) \lambda_f(m)
\end{equation}
(see also \cite[Section 3]{SarnakUbis} for the best current result). 
The choice of sieve weights $w_n$ allows us to introduce a bilinear structure, 
and to bound (\ref{eq:wnwn'sums}) on average by picking out 
the cancellations 
in (\ref{eq:scp})
coming from the large primes, while ignoring the small primes (which are a source of
problem due to the size fluctuation of $|\lambda_f(n)|$ when $n$ has many small prime
factors).   
To obtain a good lower bound for the second, non-oscillating, sum in \eqref{eq:wnwn'sums} for a positive proportion of $x$ we
compare the first and second moments, 
$$
\sum_{X \leq n \leq 2X} w_n' \text{ and }
\sum_{X \leq n \leq 2X} {w_n'}^2.
$$
The computation of these moments relies on the close resemblence of $w_n'$ to a multiplicative function
supported on the large primes. 

Finally Theorem~\ref{th:signchanges} does not need to be limited to the full modular group as long as coefficients of the form are real. For holomorphic forms $f$ with complex multiplication, Serre has shown that
\cite{Serre81} 
$$
\sum_{\substack{ p \leq x \\ \lambda_f(p) = 0}} \frac{1}{p} = \frac 12 \log\log x + O(1).
$$
A variant of Theorem~\ref{th:signchanges} could be used to show that, conditioned on the set of those 
$n \leq x$ for which $\lambda_f(n) \neq 0$, there is a positive proportion of
sign changes. We end our discussion by posing the following problem.

\begin{problem}
Let $f$ be a Hecke Maass cusp form of eigenvalue $ > \tfrac 14$. 
Is it possible to show that the coefficients $\lambda_f(n)$ are not lacunary? Precisely is it possible to show that
$$
\sum_{\lambda_f(p) = 0} \frac{1}{p} < \infty?
$$
\end{problem}
We remark that one could work out explicitly the proportion of sign 
changes that Theorem~\ref{th:signchanges} yields, although we expect that doing so is rather hard. It might also be possible to apply our techniques to study sign changes of the error term in the number of representation of $n$ by a quadratic form in $2k$ variables. However to obtain, when applicable, a good result (such as a positive proportion of sign changes) would require adapting some of the techniques of Selberg designed to study zeros of linear combinations of $L$-functions, which would increase the complexity of the proof. 

 The plan of the paper is as follows: In Section~\ref{sec:MultfProp} we prove Theorem~\ref{th:longintf}  using properties of general multiplicative functions. In Section~\ref{se:pfofsignchanges} we formulate three propositions and show how our main result, Theorem~\ref{th:signchanges} follows from them. These three propositions are then proved in Sections~\ref{se:proofmoments}--\ref{se:prooflowbound}. 
\\ 
\begin{rem} While this paper was being refereed,  
we have been able to obtain (by completely different means) a rather general result on multiplicative function in short intervals~\cite{MR} which implies among other things that a multiplicative function $f: \mathbb{N} \rightarrow \mathbb{R}$ has a positive proportion of sign changes 
as soon as $f(m) < 0$ for some integer $m$ and $f(n) \neq 0$ for a positive proportion of integers $n$. 
This recovers Theorem~\ref{th:signchanges} in the holomorphic case, but not in the case of Maass forms, 
since for a Maass form it is currently not ruled out that $\lambda_f(n) = 0$ for almost all integers $n$. 
In addition as pointed out by the anonymous referee, the method developed in this paper is general and will
work for any multiplicative function satisfying reasonable estimates for the associated shifted 
convolution problem, which is especially interesting when the function vanishes on many primes. 
\end{rem}

\section{Proof of Theorem~\ref{th:longintf}}
\label{sec:MultfProp}
Theorem~\ref{th:longintf} follows quickly from the following lemma.
\begin{lem}
\label{le:multfunc+-1}
Let $K, L \colon \R_+ \to \R_+$ be such that $K(x) \to 0$ and $L(x) \to \infty$ for $x \to \infty$. Let $g: \N \to \{-1, 0, 1\}$ be a multiplicative function such that, for every $x \geq 2$,
\[
\sum_{\substack{p \geq x \\ g(p)=0}} \frac{1}{p} \leq K(x) \quad \text{and} \quad \sum_{\substack{p \leq x \\ g(p) = -1}}\frac{1}{p} \geq L(x).
\]
Then
\[
\begin{split}
|\{n \leq x \colon g(n) = 1\}| &= (1+o(1))|\{n \leq x \colon g(n) = -1\}|\\
&= \left(\frac{1}{2}+o(1)\right) x \prod_{p \in \Pb} \left(1-\frac{1}{p}\right) \left(1 + \frac{|g(p)|}{p} + \frac{|g(p^2)|}{p^2} + \dotsb \right),
\end{split}
\]
where $o(1) \to 0$ when $x \to \infty$ and the rate of convergence depends only on $K(x)$ and $L(x)$ but not on $g$.
\end{lem}

For the proof of this lemma, we use two results concerning averages of multiplicative functions. The first lemma (see \cite{HT91} though a good enough result for our purposes would also quickly follow from Hal\'asz's theorem) shows that a real-valued multiplicative function $g\colon \N \to [-1, 1]$ is small on average unless it ``pretends'' to be the constant function $1$.

\begin{lem}
\label{le:Halreal}
There exists an absolute positive constant $C$ such that, for any multiplicative function $g$ with $-1 \leq g(n) \leq 1$, one has
\[
\sum_{n \leq x} g(n) \leq C \cdot x \exp\left(-\frac{1}{4} \sum_{p \leq x} \frac{1-g(p)}{p}\right)
\]
for all $x \geq 2$.
\end{lem}

The second lemma concerns average of a positive-valued multiplicative function.
\begin{lem}
\label{le:realDel}
Let $\varepsilon > 0$ and let $K \colon \R_+ \to \R_+$ be such that $K(x) \to 0$ for $x \to \infty$. There exists a positive constant $x_0$ (depending on $\varepsilon$ and $K(x)$) such that if $g \colon \N \to [0,1]$ is any multiplicative function for which $\sum_{p > x} \frac{1-g(p)}{p} \leq K(x)$ for all $x \geq 1$, then
\[
\left|\sum_{n \leq x} g(n) - x M(g)\right| < \varepsilon x \quad \text{for all $x \geq x_0$},
\]
where
\[
M(g) = \prod_{p \in \Pb} \left(1-\frac{1}{p}\right)\left(1+\frac{g(p)}{p} + \frac{g(p^2)}{p^2} + \dotsb\right).
\]
\end{lem}
\begin{proof}
A version which is non-uniform in $g$ can be found in \cite[Theorem 11 in Section I.3.8]{Tenenbaum95}. However it is easy to see from the proof that the claimed uniformity holds.
\end{proof}

\begin{proof}[Proof of Lemma~\ref{le:multfunc+-1}]
Lemma~\ref{le:Halreal} implies that $\sum_{n \leq x} g(n) = o(x)$ and Lemma~\ref{le:realDel} that
\[
\sum_{n \leq x} |g(n)| = (1+o(1)) x \prod_{p \in \Pb} \left(1-\frac{1}{p}\right) \left(1 + \frac{|g(p)|}{p} + \frac{|g(p^2)|}{p^2} + \dotsb \right),
\]
so the claim follows since $g(n)$ only takes values in $\{-1, 0, 1\}$.
\end{proof}

Lemma~\ref{le:multfunc+-1} immediately implies the following result on signs of multiplicative functions.
\begin{lem}
\label{le:multfunc><0}
Let $K, L \colon \R_+ \to \R_+$ be such that $K(x) \to 0$ and $L(x) \to \infty$ for $x \to \infty$. Let $g: \N \to \R$ be a multiplicative function such that, for every $x \geq 2$,
\[
\sum_{\substack{p \geq x \\ g(p)=0}} \frac{1}{p} \leq K(x) \quad \text{and} \quad \sum_{\substack{p \leq x \\ g(p) < 0}}\frac{1}{p} \geq L(x).
\]
Then
\[
\begin{split}
|\{n \leq x \colon g(n) > 0\}| &= (1+o(1))|\{n \leq x \colon g(n) < 0\}|\\
&= \left(\frac{1}{2}+o(1)\right) x \prod_{p \in \Pb} \left(1-\frac{1}{p}\right) \left(1 + \frac{h(p)}{p} + \frac{h(p^2)}{p^2} + \dotsb \right),
\end{split}
\]
where $h(n)$ is the characteristic function of those $n$ for which $g(n) \neq 0$, and $o(1) \to 0$ when $x \to \infty$ and the rate of convergence depends only on $K(x)$ and $L(x)$ but not on $g$.
\end{lem}
\begin{proof}
Apply Lemma~\ref{le:multfunc+-1} to the multiplicative function which is $0$ for those $n$ for which $g(n) = 0$ and $g(n)/|g(n)|$ for those $n$ for which $g(n) \neq 0$.
\end{proof}

To prove the long interval result we still need the following lemma.
\begin{lem}
\label{le:posproplam><0}
One has $\lambda_f(p) < 0$ for a positive proportion of primes.
\end{lem}
\begin{proof}
This is a direct consequence of the Sato-Tate Conjecture, but follows also without such deep information for instance from \cite[Theorem 4(iii)]{Murty83} (even with $m=0$ there).
\end{proof}

\begin{proof}[Proof of Theorem~\ref{th:longintf}]
Theorem~\ref{th:longintf} follows immediately from Lemma~\ref{le:multfunc><0} together with non-lacunarity and Lemma \ref{le:posproplam><0}.
\end{proof}

\section{Outline of the proof of Theorem~\ref{th:signchanges} (sign changes)}
\label{se:pfofsignchanges}
Let us start by collecting some basic facts about $\lambda_f(n)$ which will recur through the argument. 
\begin{lem}
\label{le:prel}
Let $f$ be a holomorphic Hecke cusp form for the full modular group or a Maass Hecke cusp form for the full modular group. Write $\lambda_f(n)$ for the normalized Fourier coefficients. Then
\begin{enumerate}[(i)]
\item  \label{it:nonzeroprop}
\[
\#\{n \leq x: \lambda_f(n) \neq 0\} \asymp
x \prod_{\substack{p \leq x \\ \lambda_f(p) = 0}} \Big ( 1 - \frac{1}{p} \Big ) \asymp x \prod_{\substack{p \leq x \\ \lambda_f(p) = 0}} \Big ( 1 + \frac{1}{p} \Big )^{-1}.
\]
\item \label{it:KSupbound} $\lambda_f(n) \ll n^{7/64}$.
\item \label{it:secmom} For every $z \geq w \geq 2$,
\[ 
\sum_{w \leq p \leq z} \frac{|\lambda_f(p)|^2}{p} = \sum_{w \leq p \leq z} \frac{1}{p} + O_f(1).
\]
\item \label{it:firmomlow} For every large enough $y$,
$$
\sum_{\substack{y \leq p \leq 2y \\ |\lambda_f(p)| \geq 1/2}} |\lambda_f(p)| \geq \frac{y}{10\log y}.
$$
\end{enumerate}
\end{lem}
\begin{proof}
\begin{enumerate}[(i)]
\item The second asymptotic equality is trivial, so it is enough to prove the first one. The upper bound is an immediate consequence of an upper bound sieve and the lower bound follows for instance from \cite[Theorem 1]{GrKoMa} together with Lemma~\ref{le:posproplam><0}.
\item See \cite[Appendix 2]{Kim03}.
\item By Hecke relation, and since $\lambda_f(p)$ are real,
\[
\sum_{w \leq p \leq z} \frac{|\lambda_f(p)|^2}{p} = \sum_{w \leq p \leq z} \frac{\lambda_f(p^2)+1}{p},
\]
and the claim follows since the second symmetric power $L$-function is cuspidal automorphic (by \cite{GJ78}) and thus holomorphic and non-vanishing at $s = 1$ (see e.g. \cite[Section 5.12]{IwKo04}). 
\item
The proof is very similar to~\cite[Lemma 4.1]{Kowalski}. 
Let us first note that 
\[
\frac{1}{8}\left(1 + (x^2-1) - (x^4-3x^2+1)\right) = -\frac{x^4}{8} + \frac{x^2}{2} - \frac{1}{8} \leq 
\begin{cases} 
0 & \text{if $|x| \leq 1/2$;} \\
\frac{1}{2} & \text{if $|x| > 1/2$},
\end{cases}
\]
so that
\[
\begin{split}
&\sum_{\substack{y \leq p \leq 2y \\ |\lambda_f(p)| \geq 1/2}} |\lambda_f(p)| \geq \frac{1}{8} \sum_{\substack{y \leq p \leq 2y}} 1 + (\lambda_f(p)^2-1) - (\lambda_f(p)^4-3\lambda_f(p)^2+1) \\
&= \frac{1}{8} \sum_{\substack{y \leq p \leq 2y}} 1 + \lambda_f(p^2) - \lambda_f(p^4) = \left(\frac{1}{8} + o(1)\right) \frac{y}{\log y},
\end{split}
\]
by holomorphy and non-vanishing of second and fourth symmetric power $L$-functions (see
\cite{Kim03} and
~\cite[Page 194]{KimShahidi02}).
\end{enumerate}
\end{proof}

In this section we prove  Theorem~\ref{th:signchanges} assuming propositions which we will prove in Sections \ref{se:proofmoments}--\ref{se:prooflowbound}. As described in the introduction, the basic idea is to show incompatible bounds for mollified short interval sums~(\ref{eq:wnwn'sums}). Let us start by fixing our choices of $w_n$ and $w_n'$ and the associated notation. 

Fix $X \geq 1$ and set $y = X^{\delta}$ for some small enough $\delta$. Moreover set
$$
\mathcal{D}^{+} = \{ d = p_1 \cdots p_r : p_r < \cdots < p_1 \ , \ p_m \leq y_m \text{ for all odd $m$} \},
$$
where $y_m = y^{\frac{1}{2}(1-\gamma^2) \gamma^{m-1}}$ with a parameter $\gamma \in (0, 1)$, so that $\mathcal{D}^+ \subset [1, y]$.  Now $\rho^{+}(n) = \sum_{d | n , d \in \mathcal{D}^+} \mu(d)$ are Brun's sieve weights, so that writing $P^-(n)$ for the smallest prime factor of $n$, we have
$$
\mathbf{1}_{P^-(n)>y} \leq \rho^{+}(n).
$$
In particular $\rho^{+}(n) \geq 0$ for all $n$. 
We then define
$$
w_n := \sum_{\substack{a b = n \\ a \leq y, \lambda_f(a) \neq 0  \\ (a,b) = 1}} \mu^2(a) \rho^{+}(b) |\lambda_f(b)|
$$
and 
$$
w_n' := \sum_{\substack{a b = n \\ a \leq y,  \lambda_f(a) \neq 0 \\ P^{-}(b) > y }} \mu^2(a) |\lambda_f(b)|.
$$
Write also
\[
k(x) = \prod_{\substack{p \leq x \\ \lambda_f(p) = 0}} \Big ( 1 + \frac{1}{p} \Big ),
\quad \text{so that} \quad
\sum_{\substack{n \leq x \\ \lambda_f(n) \neq 0}} 1 \asymp X/k(x)
\]
by Lemma~\ref{le:prel}\eqref{it:nonzeroprop}.

Our goal is to compare the sums
\begin{equation} \label{eq:behavior}
\sum_{x \leq n \leq x + hk(x)} \text{sgn}(\lambda_f(n)) w_n \quad \text{and} \quad
\sum_{\substack{x \leq n \leq x + hk(x) \\ \lambda_f(n) \neq 0}} w_n'
\end{equation}
for a large constant $h$. 
Note that $w_n' \leq w_n$. Therefore if the first sum is smaller in absolute value then we have detected a sign change. Note also that $w_n'$ consists of only one term. 
We first need two results on moments of $w_n'$ and $w_n$. From these we will then deduce the behavior of the sums in (\ref{eq:behavior}) and the main result will then follow shortly. 
We will prove the propositions below in sections~\ref{se:proofmoments} and~\ref{se:wnmom} respectively. 
\begin{prop} \label{pr:moments}
We have, for all $x > 0$ large enough,
$$
\sum_{X \leq n \leq 2X} {w'_n} \gg X \prod_{\substack{p \leq X  \\ \lambda_f(p) = 0}} \Big (1  - \frac{1}{p} \Big ) \quad \text{and} \quad \sum_{X \leq n \leq 2X}
{w_n'}^2 \ll X \prod_{\substack{p \leq X \\ \lambda_f(p) = 0}}
\Big ( 1 - \frac{1}{p} \Big ). 
$$
\end{prop}

Note that the second moment bound is trivial when $f$ is holomorphic since if
$P^-(b) > y = X^{\delta}$ then 
$|\lambda_f(b)| \leq \tau(b) \ll 2^{1/\delta}$ and consequently
$w_n' \ll 2^{2/\delta}$. Also for holomorphic forms $\sum_{\lambda_f(p) = 0} 1/p < \infty$
by a result of Serre.  
However the second moment estimate is less trivial for Maass forms.
We also establish
a bound for the second moment of $w_n$ which constitutes the technically 
hardest part of our proof. 
\begin{prop} \label{pr:upperbound2}
We have,
$$
\sum_{X \leq n \leq 2X} w_n^2 \ll X \prod_{\substack{p \leq X \\ \lambda_f(p) = 0}}
\Big (1 - \frac{1}{p} \Big ). 
$$
\end{prop}
From Proposition \ref{pr:upperbound2} and an estimate for the shifted convolution
problem of $\lambda_f(n)$ we deduce the following proposition. 
\begin{prop} 
\label{pr:upperbound}
There exist positive constants $C$ and $\varepsilon$ such that, uniformly in $K > 0$ and $h \leq X^\varepsilon$, one has, for at least proportion $(1 - 1/K^2)$ of $x \in [X, 2X]$, 
$$
\bigg | \sum_{x \leq n \leq x + hk(x)} \text{sgn}(\lambda_f(n)) w_n \bigg |
\leq C K \sqrt{h}. 
$$
\end{prop}

On the other hand from Proposition \ref{pr:moments} we deduce the following
complimentary result. 
\begin{prop}\label{pr:lowerbound}
For any $h \geq 1$, there is a positive proportion of $x \in [X, 2X]$
such that
\begin{equation}
\label{eq:shortlowbound}
\sum_{x \leq n \leq x + h k(X)} w_n' \gg h.
\end{equation}
\end{prop}

\begin{proof}

Let $\varepsilon$ be a small positive constant and let $\mathcal{H}_1$ be the set of square-free integers $n \in [X, 2X]$ for
which $w'_n \geq \varepsilon$. Then we have by the first part of Proposition~\ref{pr:moments},
\[
\sum_{n \in \mathcal{H}_1} w'_n \geq \sum_{X \leq n \leq 2X} w'_n - \varepsilon \sum_{\substack{X \leq n \leq 2X \\ w'_n \neq 0}} 1 \gg X \prod_{\substack{p \leq X \\ \lambda_f(p) = 0}} \left(1-\frac{1}{p}\right).
\]
Hence by Cauchy-Schwarz and the second part of Proposition~\ref{pr:moments},
\begin{equation}
\label{eq:H1lowbound}
|\mathcal{H}_1| \geq \frac{(\sum_{n \in \mathcal{H}_1} w'_n)^2}{\sum_{n} {w'_n}^2} \geq \frac{\delta X}{k(x)}
\end{equation}
for some $\delta > 0$.
Consider now the following sets 
\begin{align*}
\mathcal{U}_0 & := \{ X \leq x \leq 2X: |[x, x + h k(X)] \cap \mathcal{H}_1|
\leq \delta h / 2 \} \\
\mathcal{U}_1 & := \{ X \leq x \leq 2X: \frac{h}{\delta^2} \geq |[x, x + h k(X)] \cap \mathcal{H}_1| > \delta h / 2 \} \\
\mathcal{U}_2 & := \{ X \leq x \leq 2X: |[x, x + h k(X)] \cap \mathcal{H}_1|
\geq \frac{h}{\delta^2} \}.
\end{align*}
Then by \eqref{eq:H1lowbound}
\begin{align*}
\delta h X & \leq \sum_{X \leq x \leq 2X} |[x, x + hk(x)) \cap \mathcal{H}_1| + O(1)\\ &  \leq \left(\sum_{x \in \mathcal{U}_0} + \sum_{x \in \mathcal{U}_1} + \sum_{x \in \mathcal{U}_2} \right) |[x, x + hk(x)] \cap \mathcal{H}_1| + O(1) \\
& \leq X \cdot \delta h/2 + |\mathcal{U}_1| \cdot \frac{h}{\delta^2} + \sum_{x \in \mathcal{U}_2} |[x, x + h k(X)] \cap \mathcal{H}_1| + O(1)
\end{align*}
Notice that $n \in \mathcal{H}_1$ implies that $w_n' \geq \varepsilon$ and hence
$\lambda_f(n) \neq 0$. In particular, 
\begin{align*}
\sum_{X \leq x \leq 2X} |[x, x + h k(X)] \cap \mathcal{H}_1|^2 
&  \ll h k(X) \sum_{\substack{X \leq n \leq 2X \\ \lambda_f(n) \neq 0}} 1 + h k(X) \sum_{0 \neq |\Delta| \leq h k(X)} \sum_{\substack{X \leq n \leq 2X \\
\lambda_f(n) \neq 0 \\ \lambda_f(n + \Delta) \neq 0}} 1 \\
& \ll h^2 X
\end{align*}
by a result of Henriot \cite{Henriot, Henriot2}, since $\mathbf{1}_{\lambda_f(n) \neq 0}$ is a multiplicative function of $n$. In addition it follows from this that $|\mathcal{U}_2| \leq C \delta^4 X$ for some large absolute constant $C > 0$. Applying Cauchy-Schwarz and the previous two bounds, we get
\begin{align*}
\sum_{x \in \mathcal{U}_2} |[x, x + h k(X)] \cap \mathcal{H}_1| 
& \leq |\mathcal{U}_2|^{1/2} \cdot \Bigg ( \sum_{X \leq x \leq 2X} |[x, x + h k(X)] \cap \mathcal{H}_1|^2 \Bigg )^{1/2} \\ 
& \leq C_0 \delta^2 h X
\end{align*}
for some large absolute constant $C_0$. Combining the above inequalities it follows that
$$
\delta h X \leq X \delta h / 2 + |\mathcal{U}_1| \cdot \frac{h}{\delta^2}
+ C_0 \delta^2 X + O(1)
$$
Taking $\delta$ small enough but fixed, we conclude that
$|\mathcal{U}_1| \gg X$.
The claim now follows since the desired lower bound
\eqref{eq:shortlowbound} holds for every $x \in \mathcal{U}_1$ by the definitions of $\mathcal{U}_1$ and $\mathcal{H}_1$
\end{proof}

We are now ready to prove Theorem~\ref{th:signchanges}
which follows from combining Proposition \ref{pr:upperbound} and Proposition
\ref{pr:lowerbound}.  

\begin{proof}[Proof of Theorem~\ref{th:signchanges}]
According to Proposition \ref{pr:lowerbound}
we have
$$
 \sum_{x \leq n \leq x + hk(x)} w_n'  \geq c h
$$
for a positive proportion $\delta$ of $x \in [X, 2X]$, with
$c > 0$ an absolute constant. 
However, according to Proposition \ref{pr:upperbound} we have
$$
\bigg | \sum_{x \leq n \leq x + h k(x)} \text{sgn}(\lambda_f(n)) w_n \bigg |
\leq \frac{C}{\delta} \sqrt{h}
$$
for a proportion of at least $1 - \delta^2 > 1 - \delta$ of $x \in [X, 2X]$. Therefore once $h$ is large enough but fixed
(larger than $(C / (\delta c))^2$), we will have a positive
proportion of $x \in [X, 2X]$ for which
$$
\bigg | \sum_{x \leq n \leq x + hk(x)} \text{sgn}(\lambda_f(n)) w_n \bigg |
< \sum_{x \leq n \leq x + hk(x)} w_n'.
$$
For every such $x$ a sign change of $\lambda_f(n)$ must occur in the interval $[x, x+hk(x)]$ since
$w_n \geq w_n' \geq 0$ for every $n$. Hence there are $\gg X/k(X)$ sign changes, and the claim follows from Lemma~\ref{le:prel}\eqref{it:nonzeroprop}.
\end{proof}

\section{Proof of Proposition~\ref{pr:moments}}
\label{se:proofmoments}
We will use the following general bound for mean-values of multiplicative
functions.
\begin{lem}
\label{le:MVmultfunct}
Let $g \colon \mathbb{N} \to [0, \infty)$ be a multiplicative function. Let $A, B$ be constants such that 
for all $y \geq 1$,
\[
\sum_{p \leq y} g(p) \log p \leq Ay
\quad \text{and} \quad
\sum_{p\in \mathbb{P}} \sum_{\nu \geq 2} \frac{g(p^\nu)}{p^\nu} \log p^\nu \leq B.
\]
Then, for $x \geq 1$,
\[
\frac{1}{x} \sum_{n \leq x} g(n) \ll
(A+B+1) \prod_{p \leq x} \left(1-\frac{1}{p} + \sum_{\nu \geq 1} \frac{g(p^\nu)}{p^\nu}\right),
\]
where the implied constant is absolute.
\end{lem}
\begin{proof}
See \cite[Theorem III.3.5]{Tenenbaum95}.
\end{proof} 

We will start by proving the upper bound stated in Proposition~\ref{pr:moments}. 
We note that $w_n'^2 \leq g(n)$ where $g(n)$ is a multiplicative function such that
$$
g(p^\nu) = \begin{cases}
1 & \text{ if } p \leq y, \nu = 1 \text{ and } \lambda_f(p) \neq 0 \\
0 & \text{ if } p \leq y \text{ and } \nu > 1 \text { or } \lambda_f(p) = 0 \\
|\lambda_f(p^\nu)|^2 & \text{ if } p > y
\end{cases} 
$$
By Lemma~\ref{le:MVmultfunct} and Lemma~\ref{le:prel}\eqref{it:KSupbound}-\eqref{it:secmom},
\[
\begin{split}
\sum_{X \leq n \leq 2X} {w_n'}^2 &\ll X \prod_{\substack{p \leq y \\ \lambda_f(p) = 0}} \Big ( 1 - \frac{1}{p} \Big )
\prod_{\substack{y < p \leq X}} \Big (1 -\frac{1}{p} + \sum_{\nu \geq 1} \frac{|\lambda_f(p^\nu)|^2}{p^\nu} \Big ) \ll  X \prod_{\substack{p \leq X \\ \lambda_f(p) = 0}} \Big ( 1 - \frac{1}{p} \Big ).
\end{split}
\]
We now focus on the lower bound for the first moment of $w_n'$.  
The term $w'_n$ contains at most one summand. Therefore, 
\begin{align*}
\sum_{X \leq n \leq 2X} w'_n & = \sum_{\substack{X \leq a b \leq 2X \\ a \leq y, \lambda_f(a) \neq 0 \\ P^{-}(b) > y}}
\mu^2(a) |\lambda_f(b)| \geq \sum_{2X/y \leq p \leq 2X} |\lambda_f(p)| \sum_{\substack{X / p \leq a \leq 2X / p \\ \lambda_f(a) \neq 0}} |\mu(a)|^2 \\
& \gg X \prod_{\substack{p \leq X \\ \lambda_f(p) = 0 }} \Big ( 1 - \frac{1}{p} \Big ) 
\sum_{\substack{2X / y \leq p \leq 2X}} \frac{|\lambda_f(p)|}{p} 
\gg X \prod_{\substack{p \leq X \\ \lambda_f(p) = 0 }} \Big ( 1 - \frac{1}{p} \Big ) 
\sum_{\substack{2X / y \leq p \leq 2X}} \frac{1}{p}
\end{align*}
by Lemma~\ref{le:prel}\eqref{it:firmomlow}. Hence
\[
\sum_{X \leq n \leq 2X} w'_n \gg X \prod_{\substack{p \leq X \\ \lambda_f(p) = 0 }} \Big ( 1 - \frac{1}{p} \Big ) \log \frac{\log 2X}{\log (2X/y)} \gg X \prod_{\substack{p \leq X \\ \lambda_f(p) = 0 }} \Big ( 1 - \frac{1}{p} \Big ).
\]


\section{Proof of Proposition~\ref{pr:upperbound2}}
\label{se:wnmom}
Recall that $\rho^+(n)$ are upper bound Brun sieve weights, so that (compare with \cite[Section 6.2]{Opera})
\[
\begin{split}
\rho^{+}(n) - \mathbf{1}_{P^-(n)>y} &= \sum_{r \geq 0} \sum_{\substack{n = p_1 \cdots p_{2r + 1} d \\ p_1 > \cdots > p_{2r + 1} > y_{2r+1} \\ p_{2\ell + 1} \leq y_{2\ell+1} \ (\forall \ell < r) \\ P^-(d) > p_{2r+1}}} \mu(d) \leq \sum_{r \geq 0} \mathbf{1}_{P^-(n)>y_{2r+1}} \mathbf{1}_{\omega(n) \geq 2r+1} \sum_{k \mid n} \mu(k)^2 \\
&\leq \sum_{r \geq 0} \mathbf{1}_{P^-(n)>y_{2r+1}}  2^{-(2r+1)} 2^{2\omega(n)}.
\end{split}
\]

Hence
$$
w_n = \sum_{\substack{a b = n \\ a \leq y, \lambda_f(a) \neq 0  \\ (a,b) = 1}} \mu^2(a) \rho^{+}(b) |\lambda_f(b)| \leq w_n' + w_n''
$$
where $w_n'$ is as before and 
\[
\begin{split}
w_n'' &:= \sum_{r \geq 0} 2^{-2r} \sum_{\substack{a b = n \\ a \leq y, \lambda_f(a) \neq 0  \\ P^-(b) > y_{2r+1}, (a,b) = 1}} \mu^2(a) |\lambda_f(b)| 2^{2\omega(b)} .
\end{split}
\]
Note that $w_n^2 \leq 2 w_n'^2 + 2w_n''^2$. 
By Proposition~\ref{pr:moments} it is enough to consider $\sum w_n''^2$. 
In the sum over divisors of $n$ in the definition of $w_n''$ we write $a = c d$ with $P^+(c) \leq y_{2r + 1}$ and $P^-(d) > y_{2r + 1}$. This allows us to re-write the sum over $a b = n$ as follows
\begin{align} \label{eq:simplebound}
\sum_{\substack{ c d b = n \\ cd \leq y , \lambda_f(c d) \neq 0 \\
P^-(bd) > y_{2r + 1} \\ P^+(c) \leq y_{2r + 1}}} \mu^2(c d) |\lambda_f(b)|
4^{\omega(b)} & \leq \sum_{\substack{c \ell = n \\ P^+(c) \leq y_{2 r + 1} \\ \lambda_f(c) \neq 0}} \mu^2(c) g_r(\ell) := G_r(n)
\end{align}
where $g_r(\ell)$ is a multiplicative function such that 
\[
\begin{split}
g_r(p^{\nu}) &:= \sum_{\substack{ p^{\nu} = d b \\ P^-(b d) > y_{2r + 1}}} \mu^2(d) |\lambda_f(b)| 4^{\omega(b)} \\
&= \begin{cases}
1 + 4 |\lambda_f(p)| & \text{ if } p \geq y_{2r + 1} \text{ and } \nu = 1 \\
4 |\lambda_f(p^{\nu})|+ 4 |\lambda_f(p^{\nu-1})| & \text{ if } p \geq y_{2r + 1} \text{ and } \nu \geq 2\\
0 & \text{ otherwise} 
\end{cases} .
\end{split}
\]
%
By Cauchy-Schwarz and (\ref{eq:simplebound}),
$$
\sum_{X \leq n \leq 2X} w_n''^2 \ll \sum_{r \geq 0} 2^{-2r}
\sum_{X \leq n \leq 2X} G_r(n)^2
$$
Note that $G_r(n)^2$ is also multiplicative, and 
\[
G_r(p)^2 = 
\begin{cases} 
(1+4|\lambda_f(p)|)^2 \leq 4 + 64|\lambda_f(p)|^2 & \text{if $p > y_{2r+1}$;} \\
0 &\text{if $p \leq y_{2r+1}$ and $\lambda_f(p) = 0$;}\\
1 &\text{if $p \leq y_{2r+1}$ and $\lambda_f(p) = 1$,}
\end{cases}
\]
and $G_r(p^\nu)^2 \leq 64 p^{7\nu/32}$ by Lemma~\ref{le:prel}\eqref{it:KSupbound}.
By Lemma \ref{le:MVmultfunct} and Lemma~\ref{le:prel}\eqref{it:secmom} we get
\begin{align*}
\sum_{X \leq n \leq 2X} G_r(n)^2 & \ll X \prod_{\substack{p \leq y_{2r + 1} \\
\lambda_f(p) = 0 }} \Big (1 - \frac{1}{p} \Big ) \cdot \prod_{y_{2r + 1}
\leq p \leq X} \Big ( 1 + \frac{4 + 64 |\lambda_f(p)|^2}{p} \Big ) \\
& \ll X \prod_{\substack{p \leq X \\ \lambda_f(p) = 0}} \Big ( 1 - \frac{1}{p}
\Big ) \cdot \Big ( \frac{\log X}{\log y_{2r + 1}} \Big )^{69}
\end{align*}
Combining the above equations and recalling that
$y_r = X^{\delta/2 (1 - \gamma^2) \gamma^{m-1}}$ we obtain
\begin{align*}
\sum_{X \leq n \leq 2X} w_n''^2 \ll X \prod_{\substack{p \leq X \\ 
\lambda_f(p) = 0}} \Big ( 1- \frac{1}{p} \Big ) 
\sum_{r \geq 0} 2^{-2r} \Big ( \frac{1}{(1 - \gamma^2) \delta \gamma^{2r}}
\Big )^{69} 
\end{align*}
Picking $\gamma = 2^{-1/100}$ we see that the sum over $r$ is $O(1)$. Combining the previous bounds then leads to
$$
\sum_{X \leq n \leq 2X} w_n^2 \ll X \prod_{\substack{p \leq X \\ \lambda_f(p) = 0}} \Big ( 1 - \frac{1}{p} \Big )
$$
as was claimed. 

\section{Proof of Proposition~\ref{pr:upperbound}}
\label{se:prooflowbound}
An important role in the proof of Proposition~\ref{pr:upperbound} is played by estimates for shifted convolution sums.
\begin{lem}
\label{le:shiftconvsum}
There exists a small $\delta > 0$ such that, 
uniformly in $a, b, A,B,h \leq x^{\delta}$,
\begin{equation}
\label{eq:shiftconvsum}
\sum_{\substack{x \leq aAm, bBn \leq 2x \\ aAm - bBn = h}} \lambda_f(A m) \lambda_f(B n) \ll x^{1 - \delta}.
\end{equation}
\end{lem}
\begin{proof}
Notice first that, for any integer $r$,
\[
\begin{split}
f_r(z) &:= \sum_{\substack{n \\ r | n}} \lambda_f(n) n^{(\kappa-1)/2} e(n z) = 
\sum_{n} \lambda_f(n) n^{(\kappa-1)/2} \cdot \bigg ( \frac{1}{r} \sum_{0 \leq j < r}
e(j n / r) \bigg ) e(n z) \\
&=
\frac{1}{r}
\sum_{0 \leq j < r} f(z + j/r)
\end{split}
\]
is a cusp form of weight $\kappa$ and level $r^2$ (see for instance~\cite[Proof of Proposition 14.19]{IwKo04} for a related argument, noting that in the matrix identity below~\cite[(14.52)]{IwKo04} the upper right corner of the right-most matrix should be $d^2u/r$). 

We can re-write the left-hand side of \eqref{eq:shiftconvsum} as
$$
\sum_{\substack{x \leq am, bn \leq 2x \\ am - bn = h}} \lambda_{f_A}(m) \lambda_{f_B}(n)
$$
where $\lambda_{f_A}(n) = \lambda_f(n)$ if $A | n$ and $\lambda_{f_A}(n) = 0$ otherwise. 
By a result of Harcos~\cite[Theorem 1]{Harcos03} this is indeed bounded by $x^{1-\delta}$. The implicit constant in Harcos's result depends polynomially on the level and weight of the form (see \cite[Addendum in end of Section 1]{Harcos03}) and furthermore it depends on the implicit constants in Wilton's estimate
$$
\sum_{n \leq x} \lambda_{f_A}(n) e(n \alpha) \ll \sqrt{x} \log x
$$
and in the Rankin Selberg bound
$$
\sum_{n \leq x} |\lambda_{f_A}(n)|^2 \ll x.
$$
Since $|\lambda_{f_A}(n)| \leq |\lambda_f(n)|$ the second sum is trivially less than $\sum_{n \leq x} |\lambda_f(n)|^2 \ll x$ with no dependence on $A$. On the other hand, the first sum is by Wilton's bound,
$$
\sum_{n \leq x} \lambda_{f_A}(n) e(n \alpha) = \frac{1}{A} \sum_{j < A} \Big ( 
\sum_{n \leq x} \lambda_f (n) e(n (\alpha + j / r)) \Big ) \ll \sqrt{x} \log x
$$
also with no dependence on $A$, since Wilton's bound $\sum_{n \leq x} \lambda_f(n) e(n\alpha) \ll \sqrt{x} \log x$ is uniform in $\alpha \in \mathbb{R}$. 
\end{proof}
This lemma can be used to prove the following mean square estimate in short intervals.
\begin{lem}
There are absolute positive constants $c$ and $\eta$ such that uniformly in $h \leq X^\eta$, 
$$
\int_X^{2X} \bigg | \sum_{x \leq n \leq x + h k(X)} \text{sgn}(\lambda_f(n)) w_n \bigg |^{2} dx \leq c^2 h X
$$
\end{lem}
\begin{proof}
We start by expanding the square obtaining sums over $n_1$ and $n_2$, writing then $n_2-n_1 = r$, and interchanging the integration and summations. This shows that the left hand side equals
\[
\begin{split}
&\sum_{X \leq n_1 \leq 2X} \sum_{|r| \leq h k(X)} (hk(X)-|r|) \text{sgn}(\lambda_f(n_1)) w_{n_1} \text{sgn}(\lambda_f(n_1+r)) w_{n_1+r}\\
&\qquad + O\left(hk(X)\left(\sum_{\substack{|n_1 - X| \leq h k(X) \\ |n_2 - X| \leq h k(X)}} |w_{n_1}w_{n_2}| + \sum_{\substack{
|n_1 - 2X| \leq h k(X) \\ |n_2 - 2X| \leq h k(X)}} |w_{n_1}w_{n_2}|\right)\right).
\end{split}
\]
The error terms contribute at most
\[
(h k(X))^3 \max_{n \leq 3X} |w_n|^2 \ll X^{4\eta} X^{7/32} \ll X^{1/2}
\]
by Lemma~\ref{le:prel}\eqref{it:KSupbound}, so we can concentrate on the main term which, by definition of $w_n$ equals
\begin{align*}
& \sum_{0 < |r| \leq hk(X)} (hk(X)-|r|) \sum_{\substack{X \leq a b \leq 2X \\ a' b' - a b = r \\
(a,b) = (a',b') = 1 \\ a, a' \leq y}} \mu^2(a)\mu^2(a') \text{sgn}(\lambda_f(a))\text{sgn}(\lambda_f(a'))\rho^+(b) \rho^+(b') \lambda_f(b) \lambda_f(b') \\ & + hk(X) \sum_{\substack{X \leq n \leq 2X \\ \lambda_f(n) \neq 0}} w_n^2.
\end{align*}
By Proposition \ref{pr:upperbound2} the contribution from the diagonal term $h k(X) \sum w_n^2$ is 
bounded by $\ll h X$.  We re-write the off-diagonal terms as
\[
\begin{split}
&\sum_{0 < |r| \leq hk(X)} (hk(X)-|r|) \sum_{a, a' \leq y} \mu(a)^2 \mu(a')^2 \text{sgn}(\lambda_f(a)) \text{sgn}(\lambda_f(a')) \\
&\qquad \qquad \cdot
\sum_{\substack{X / a \leq b \leq 2X / a \\ (a,b) = 1 \\ (a', b') = 1\\  a' b' - a b = r}} \rho^+(b) \rho^+(b') \lambda_f(b) \lambda_f(b')
\end{split}
\]
We focus on the innermost sum, which we re-write in the following way (after introducing the Moebius function to remove the co-primality conditions), 
$$
\sum_{\substack{d | a \\ d' | a'}} \mu(d) \mu(d') \sum_{\substack{ X / a \leq b \leq 2X / a \\ d | b, d' | b' \\ a' b' - a b = r}} \rho^+(b) \rho^+(b') \lambda_f(b) \lambda_f(b').
$$
Furthermore expanding $\rho^+(b)$ according to its definition, we get
$$
\sum_{\substack{d | a \\ d' | a'}} \mu(d) \mu(d') \sum_{\substack{e, e' \in \mathcal{D}^+}}
\mu(e) \mu(e') \sum_{\substack{X / a \leq b \leq 2X / a \\ [d,e] | b \\ [d', e'] | b' \\
a' b' - a b = r}} \lambda_f(b) \lambda_f(b').
$$
Notice that $\mathcal{D}^+ \subset [1, y]$, and that $d,d' \leq y$ since $a, a' \leq y$. 
Therefore the inner sum is by Lemma \ref{le:shiftconvsum} bounded by $\ll X^{1 - \eta}$ for
some absolute $\eta > 0$, provided that $y = X^{\delta}$ is chosen small enough. Combining
the above equations it follows
that the contribution of the off-diagonal terms is bounded by $\ll (hk(X))^2 y^{4+\varepsilon} X^{1 - \eta}$
and this is $\ll X^{1 - \kappa}$ for some $\kappa > 0$ provided that $\delta > 0$ is chosen small enough. 
\end{proof}
\begin{proof}[Proof of Proposition~\ref{pr:upperbound}]
By previous lemma and Chebychev's inequality the measure of the set of $x \in [X, 2X]$ for which
$$
\bigg | \sum_{x \leq n \leq x + hk(x)} sgn(\lambda_f(n))w_n \bigg |
\geq c K \sqrt{h}
$$
is at most $X/K^2$. 
\end{proof}


\section{Proof of Corollary~\ref{cor:Cho}}
\label{se:pfofcor}
Write $g(n) = \text{sgn}(\lambda_f(n))$, so that $g(n)$ is a multiplicative function taking values in $\{-1, 0, 1\}$.
The upper bound
\[
\sum_{n \leq x} g(n)g(n+1) \leq (1 - c) x.
\]
follows immediately from Theorem~\ref{th:signchanges}. To obtain the lower bound
\begin{equation}
\label{eq:cholowe}
\sum_{n \leq x} g(n)g(n+1) \geq (-1 + c) x,
\end{equation}
we first note that, by~\cite[Lemma 4.1]{GhoshSarnak}, there is an even integer $b \ll_f 1$ such that $g(2^b) = 1$. Notice that the claim is trivial unless $g(2^j) \neq 0$ for all $j \leq b$. If $g(2) = -1$, the existence of such $b$ implies that there is an integer $j < b$ such that $g(2^j) = -1$ and $g(2^{j+1}) = 1$. On the other hand if $g(2) = 1$, then, since $b$ is even, there must be an integer $j < b$ such that $g(2^j) = g(2^{j+1})$. Hence in any case we find $j \ll 1$ such that $g(2^{j+1}) = g(2) g(2^j)$.

Consider now $n \equiv 2^j-1 \pmod{2^{j+1}}$. Notice that, for such $n$, $g(2n) = g(2) g(n)$ and 
\[
g(2n+2) = g(2^{j+1}) g((n+1)/2^j) = g(n+1) g(2^{j+1})/g(2^j) = g(2)g(n+1),
\]
so that
\[
g(n)g(n+1)g(2n)g(2n+1)g(2n+1)g(2n+2) = (g(n)g(n+1)g(2)g(2n+1))^2 \geq 0.
\]
Hence one of $g(n) g(n+1)$, $g(2n)g(2n+1)$ and $g(2n+1)g(2n+2)$ must be non-negative, and \eqref{eq:cholowe} follows since positive proportion of $n \leq x/2$ satisfy $n \equiv 2^j-1 \pmod{2^{j+1}}$.

\section*{Acknowledgments}
The authors would like to thank the anonymous referee for suggesting a complete re-write of the proof (a variant of which is presented
here). The referee's argument led to several improvements and to a much cleaner proof, for which the authors are extremely
grateful.  
The authors would also like to thank Andrew Granville for providing the proof of Lemma~\ref{le:multfunc+-1} and Matti Jutila for mentioning the shifted convolution problem. They are also grateful to Jie Wu and Wenguang Zhai for an e-mail discussion concerning~\cite{WuZhai13} and to Eeva Vehkalahti for helpful comments on earlier versions of this manuscript. They also would like to thank Peter Sarnak for a Lemma which was used in the previous version of the paper and for posing the question on the entropy of $\lambda_f(n)$ . Finally the authors are grateful to Brian Conrey for suggesting
the use of a mollifier at the October 2013 meeting in Oberwolfach.
\bibliography{HeckeSignChanges}
\bibliographystyle{plain}
\end{document}